\documentclass[pdflatex,sn-mathphys-num]{sn-jnl}% Math and Physical Sciences Numbered Reference Style
\usepackage{graphicx}%
\usepackage{multirow}%
\usepackage{amsmath,amssymb,amsfonts}%
\usepackage{amsthm}%
\usepackage{mathrsfs}%
\usepackage[title]{appendix}%
\usepackage{xcolor}%
\usepackage{textcomp}%
\usepackage{manyfoot}%
\usepackage{booktabs}%
\usepackage{algorithm}%
\usepackage{algorithmicx}%
\usepackage{algpseudocode}%
\usepackage{listings}%
\usepackage{bm}
\numberwithin{equation}{section}
\newtheorem{theorem}{Theorem}[section]
\newtheorem{proposition}{Proposition}[section]

\newtheorem{definition}{Definition}[section]
\newtheorem{lemma}{Lemma}[section]
\newtheoremstyle{remarkstyle}
  {3pt}      % 上方空间
  {3pt}      % 下方空间
  {\normalfont}  % 正文字体
  {}         % 缩进
  {\itshape} % 标题字体 Remark
  {.}        % 标题后面跟一个句号
  {0.5em}    % 标题与正文的间距
  {}         % 标题格式（保持默认）
\theoremstyle{remarkstyle}
\newtheorem{remark}{Remark}[section]

\newcommand\br{\mathbb{R}}

\newcommand\bbt{\mathbb{T}}
\newcommand\dd{\,\mathrm{d}}

\begin{document}

\title[Article Title]{Explicit symmetric low-regularity integrators for the semilinear Klein-Gordon equation}


\author[1]{\fnm{Zhirui} \sur{Shen}}\email{zrshen@stu.xjtu.edu.cn}

\author*[1]{\fnm{Bin} \sur{Wang}}\email{wangbinmaths@xjtu.edu.cn}

\affil[1]{\orgdiv{Mathematics and Statistics}, \orgname{Xi'an Jiaotong University}, \orgaddress{\city{Xi'an}, \postcode{710049}, \country{China}}}

%\affil[2]{\orgdiv{Mathematics and Statistics}, \orgname{Xi'an Jiaotong University}, \orgaddress{\city{Xi'an}, \postcode{710049}, \country{Country}}}


\abstract{This paper is concerned with the design and analysis of symmetric low-regularity integrators for the semilinear Klein-Gordon equation. We first propose a general symmetrization procedure that allows for the systematic construction of symmetric schemes from existing explicit (non-symmetric) integrators. Applying this procedure, we derive two novel schemes. Error analyses show that both integrators achieve their optimal convergence orders in the energy space under significantly relaxed regularity assumptions. Furthermore, the symmetry property ensures that the convergence order of a first-order symmetric scheme improves as the regularity of the exact solution increases. A numerical experiment demonstrates that the proposed second-order symmetric scheme nearly preserves the system energy over extended periods.}

\keywords{Low-regularity integrator, Error estimate, Klein-Gordon equation, Explicit symmetric methods}


\pacs[MSC Classification]{35L70, 65M12, 65M15, 65M70}

\maketitle

\section{Introduction}\label{sec1}
 In this paper, we focus on the low-regularity numerical approximation of the nonlinear Klein-Gordon equation (NKGE) on the torus $\mathbb{T}^d$ ($d=1,2,3$) of the form
\begin{equation}\label{NKGE}
    \begin{cases}
    \partial_{t t} u(x,t)-\Delta u(x,t)=f(u(x,t)),  & (x,t) \in \mathbb{T}^d  \times(0, T], \\
     u(x,0)=u^0(x),\ \partial_t u(x,0)=v^0(x), & x \in \mathbb{T}^d,
     \end{cases}
 \end{equation}
 where $x\in \mathbb{T}^d$ denotes the spatial coordinate, $t\in (0, T]$ represents time, $u(x,t)$ is a complex-valued scalar field, and the nonlinear function $f(u):\mathbb{R}\to\mathbb{R}$ is the negative derivative of a potential function $V$, i.e., $f(u)=-V'(u)$. This equation arises in various physical contexts, including magnetic-flux propagation, Bloch-wall dynamics, dislocation propagation in solids and liquid crystals, and the propagation of ultra-short optical pulses in two-level media. As a Hamiltonian PDE, the NKGE preserves the total energy
 \begin{equation}
     H(t)=\int_{\mathbb{T}^d}\left(\frac12|\partial_t u|^2+\frac12|\nabla u|^2+V(u)\right)\mathrm{d}x\equiv H(0).
 \end{equation}

The numerical approximation of the NKGE poses significant challenges, particularly for solutions with limited smoothness (low regularity). Classical numerical methods, including finite difference time domain (FDTD) schemes \cite{Yi2020Optimal,Cai2018Error,Bao2012Analysis}, trigonometric and exponential integrators \cite{Buchholz2021,Wang2019,Bao2014}, splitting methods \cite{Kropielnicka2024Effective,Baumstark2018Uniformly}, and structure-preserving geometric integrators \cite{Cohen2008,Chen2020,Hairer2006}, often struggle under these conditions. When the solution belongs to a Sobolev space $H^\alpha\times H^{\alpha-1}$ with a relatively small exponent $\alpha$, these classical approaches generally require higher regularity (e.g., $H^2(\mathbb T)\times H^1(\mathbb T)$ or higher) to achieve optimal convergence in the natural energy space $H^1(\mathbb T)\times L^2(\mathbb T)$. Consequently, for low-regularity solutions, conventional methods either fail to achieve the expected order of convergence or rely on impractically strict regularity assumptions, thereby limiting their applicability in scientific computation.

In recent years, low-regularity integrators (LRIs) have emerged as a powerful tool to overcome this regularity barrier. The pioneering work of Ostermann and Schratz~\cite{Ostermann2018} introduced a class of exponential integrators capable of handling nonsmooth solutions of the Schr{\"o}dinger equation. Building upon this foundation, numerous low-regularity integrators have since been developed for various systems \cite{Ostermann2020Cubic,Schratz2020Dirac,Ning2022mKdV}.
Regarding the NKGE, Rousset and Schratz~\cite{Rousset2021} established a general framework for low-regularity integrators that achieves second-order convergence in the energy space under the regularity assumption $(u,\partial_t u)\in H^{7/4}(\mathbb T)\times H^{3/4}(\mathbb T)$. Subsequently, Li, Schratz, and Zivcovich~\cite{Li2023} constructed a second-order low-regularity correction of the Lie splitting method, which attains second-order accuracy in the energy space under the weaker requirement $(u,\partial_t u)\in H^{1+d/4}(\mathbb T)\times H^{d/4}(\mathbb T)$ for $d=1,2,3$.

It is worth noting that none of the aforementioned low-regularity integrators are symmetric. For Hamiltonian PDEs, however, symmetric integrators constitute a particularly powerful class of methods \cite{Hairer2000Long,Li2020Linearly}. These schemes combine the advantages of exponential integration with time-reversal symmetry, a property that is crucial for ensuring superior long-time behavior, such as the near-conservation of energy. Methodologically, such integrators are often constructed using trigonometric or exponential filters and implemented explicitly in physical space, typically employing Fourier pseudospectral discretization for efficiency.
Recently, Wang and Zhao \cite{Wang2022Symmetric} proposed a symmetric low-regularity integrator for the nonlinear Klein-Gordon equation. However, the approach was restricted to the cubic nonlinearity $f(u)=u^3$ in one dimension ($d=1$). Other symmetric low-regularity integrators have been developed specifically for the nonlinear Schr{\"o}dinger equation. In \cite{alama2023symmetric},  Alama Bronsard proposed a class of implicit symmetric low-regularity integrators, while Feng, Maierhofer, and Wang \cite{Feng2024} recently introduced a fully explicit low-regularity scheme for the same equation. %By employing a multi-step construction, the latter approach overcomes the significant computational cost associated with implicit single-step symmetric low-regularity schemes, while preserving both the favorable structure-preserving behavior of symmetric methods and their low-regularity convergence properties.

Despite these advances, most existing low-regularity integrators for the NKGE lack either symmetry or computational efficiency. In particular, the symmetric scheme proposed in \cite{Wang2022Symmetric} requires a relatively involved starting procedure and is restricted to the specific nonlinearity $f(u)=u^3$ in one spatial dimension ($d=1$). Conversely, the low-regularity integrators for the NKGE presented in \cite{Rousset2021,Li2023} are non-symmetric and consequently suffer from energy drift over long-time simulations. These limitations motivate the development of simple, symmetric low-regularity integrators that combine favorable convergence properties under low regularity with superior long-time conservation behavior.

The main contributions of this work are as follows:
\begin{itemize}
\item The proposed schemes are explicit, symmetric, and low-regularity. They are easy to implement, and thanks to their symmetry, they exhibit near-conservation of energy over long time intervals.
\item We demonstrate that symmetrising the classical Lie splitting scheme yields a symmetric method, denoted sLRI1, which delivers improved convergence properties compared to the original non-symmetric version, even under low-regularity conditions.
\item Based on the second-order low-regularity corrected Lie splitting method proposed in \cite{Li2023}, we design a symmetric two-step scheme named sLRI2. This new scheme retains second-order convergence in ($H^1(\mathbb T)\times L^2(\mathbb T)$) under the same low-regularity assumption ($H^{1+d/4}(\mathbb T)\times H^{d/4}(\mathbb T)$), with its symmetry ensuring excellent long-time conservation.
\end{itemize}
    
To the best of our knowledge, this is the first symmetric low-regularity exponential integrator for the NKGE with general nonlinearities. It combines a simple two-step structure (requiring only one evaluation of the nonlinearity per time step) with provable near-conservation properties over long time intervals. Numerical experiments confirm that the new scheme outperforms existing low-regularity methods in terms of both accuracy and long-time behaviour.

The remainder of the paper is organized as follows. In Section~\ref{sec:construction}, we present the construction of the symmetric low-regularity schemes. Section~\ref{sec:convergence sLRI1} provides the stability and convergence analysis of sLRI1, while Section~\ref{sec:convergence sLRI2} is dedicated to that of sLRI2. Finally, numerical experiments in Section~\ref{sec:Numerical experiments} illustrate the performance of the new integrators.

In this paper, $C(\cdot)$ denotes a generic positive constant depending on its argument $(\cdot)$. The notation $A \lesssim B$ signifies the inequality $A \le C B$, where the constant $C > 0$ is independent of the time step $\tau$ (and the spatial mesh size $h$, if applicable).

\section{Explicit symmetric low-regularity integrators}\label{sec:construction}
We rewrite the semilinear Klein-Gordon equation into the following first-order system:
\begin{equation}\label{U_t-LU=F(U)}
    \begin{cases}\partial_tU-LU=F(U), & (x,t) \in \mathbb T  \times(0, T], \\ U(t_0)=U^0, & x \in \mathbb T ,\end{cases}
\end{equation}
with 
\begin{equation}\label{U,U^0,F(U)}
    U=\begin{pmatrix}
        u\\u_t
       
       \end{pmatrix},
    \ U^0=\begin{pmatrix}
        u^0\\v^0
       
       \end{pmatrix},
    \ F(U)=\begin{pmatrix}
        0\\f(u)
       
       \end{pmatrix},
\end{equation}
and
\begin{equation}
    L:=\begin{pmatrix}
  0&I \\
  \Delta&0
\end{pmatrix}:(H^2(\mathbb T)\cap H_0^1(\mathbb T))\times H^1(\mathbb T)\to  H_0^1(\mathbb T)\times L^2(\mathbb T)
\end{equation}
is the infinitesimal generator of group $\{e^{tL}\}_{t\in\br}$. Typically, we assume that $f\in C^2(\br;\br)$ and the Lipschtz condition
\begin{equation}\label{Lip}
    \sup_{x\in\br}|f^{(j)}(x)|\le L<\infty
\end{equation}
holds for $j=1,2$.

We begin by constructing an explicit first-order symmetric numerical scheme using the variation-of-constants formula, and then outline a general procedure for deriving symmetric numerical scheme based on this construction. Recall the definition in \cite{Hairer2006} of symmetric numerical schemes:

\begin{definition}\label{sym}
    A two-step method $y_{n+1}=\Phi_{\tau}(y_n,y_{n-1})$ with the time stepsize $\tau>0$ is symmetric or time-reversible, if exchanging $n+1\leftrightarrow n-1$ and $\tau\leftrightarrow -\tau$ leaves the method unaltered.
\end{definition}

\subsection{Construction of a specific explicit symmetric LRI}
Letting $\tau>0$ is the time stepsize and $t_n=n\tau$ for $n=0,1,\cdots,[T/\tau]$, the variation-of-constant formula to (\ref{U_t-LU=F(U)}) is (we omit $x$ for brevity from now on) 
\begin{equation}\label{vcf}
    U(t_{n}+s)=e^{s L}U(t_n)+\int_{0}^{s }e^{(s-\sigma)L}F(U(t_n+\sigma))\dd\sigma,\ \forall s\in(0,\tau].
\end{equation}
Multiplying $e^{s L}$ on both sides of (\ref{vcf}) and taking $s=\tau$, $s=-\tau$, we have
\begin{align*}
        e^{-\tau L}U(t_{n+1})&=U(t_n)+\int_{0}^{\tau }e^{-sL}F(U(t_n+s))\dd s,
     \\
        e^{\tau L}U(t_{n-1})&=U(t_n)-\int_{-\tau}^{0 }e^{-sL}F(U(t_n+s))\dd s. 
\end{align*}
Subtracting the above two equations gives
\begin{equation}\label{V_n(t)-V_n(-t)}
    e^{-\tau L}U(t_{n+1})-e^{\tau L}U(t_{n-1})=\int_{-\tau }^{\tau }e^{-sL}F(U(t_n+s))\dd s.
\end{equation}
Multiplying $e^{\tau L}$ on both sides, we have
\begin{equation}\label{main}
    U(t_{n+1})=e^{2\tau L}U(t_{n-1})+\int_{-\tau }^{\tau }e^{(\tau-s)L}F(U(t_n+s))\dd s.
\end{equation}
Expanding the right side  of (\ref{main}) by Taylor's theorem yields
\begin{align}
        U(t_{n+1})=&e^{2\tau L}U(t_{n-1})+\int_{-\tau }^{\tau }e^{(\tau-s)L}F(U(t_n+s))\dd s\nonumber\\
        =&e^{2\tau L}U(t_{n-1})+\int_{-\tau }^{\tau }e^{\tau L}F(U(t_n))\dd s\nonumber\\
        &+\int_{-\tau}^{\tau} e^{(\tau-s)L}[F(U(t_n+s))-F(e^{sL}U(t_n))]\mathrm{d}s\nonumber\\
        &+\int_{-\tau }^{\tau }e^{\tau L}\int_{0}^{s}\frac{\dd}{\dd\sigma}(e^{-\sigma L}F(e^{\sigma L}U(t_n)))\dd\sigma\dd s\nonumber\\
        :=&e^{2\tau L}U(t_{n-1})+2\tau e^{\tau L}F(U(t_n))+\mathcal R_{\tau}(t_n).
\end{align}
Therefore we obtain 
\begin{equation}\label{LRI1}
    U^{n+1}\approx e^{2\tau L}U^{n-1}+2\tau e^{\tau L}F(U^n).
\end{equation}

To compute the first step $U(t_1)$, we can use a well-known first-order low regularity approximation like Lie splitting in \cite{Li2023}, i.e.,
\[
U^1:=e^{\tau L}U^0 +\tau e^{\tau L}F(U^0).
\]
This gives us the following numerical method denoted by sLRI1:
\begin{equation}\label{sLRI1}
    \begin{aligned}
       U^{n+1}=&e^{2\tau L}U^{n-1}+2\tau e^{\tau L}F(U^n),\quad n\ge1,\\
       U^1=&e^{\tau L}U^0 +\tau e^{\tau L}F(U^0).
    \end{aligned}
\end{equation}
\begin{remark}
    Exchanging $n+1\leftrightarrow n-1$ and $\tau \leftrightarrow -\tau$ gives
\begin{equation}\label{yanzheng}
    U^{n-1}=e^{-2\tau L}U^{n+1}-2\tau e^{-\tau L}F(U^n).
\end{equation}
Multiplying both sides of (\ref{yanzheng}) by $e^{2\tau L}$ yields (\ref{sLRI1}), from which it follows that the scheme is symmetric. This observation motivates the question of whether, for an arbitrary explicit low-regularity integrator, one can apply an analogous procedure to construct a corresponding symmetric method. As shown in the next subsection, the answer to this question is affirmative.

\end{remark}

\subsection{General method for constructing symmetric scheme}
From the variation-of-constants formula (\ref{vcf}), let 
\begin{equation}
    U^{n+1}=e^{\tau L}U^n+\Psi_{\tau}(U^n)+\mathcal{R}_{\tau}^n
\end{equation}
be a prescribed exponential integrator, where $\Psi_{\tau}$ denotes the higher-order terms and $\mathcal{R}_{\tau}^n$ represents the remainder corresponding to the step number $n$. This implies
\begin{align*}
    &U^{n+1}=e^{\tau L}U^n+\Psi_{\tau}(U^n)+\mathcal R_{\tau}^n,\\
    &U^{n-1}=e^{-\tau L}U^n+\Psi_{-\tau}(U^n)+\mathcal R_{-\tau}^n.
\end{align*}
Equivalently, one has
\begin{align}
    e^{-\tau L}U^{n+1}&=U^n+e^{-\tau L}\Psi_{\tau}(U^n)+e^{-\tau L}\mathcal R_{\tau}^n,\\
    e^{\tau L}U^{n-1}&=U^n+e^{\tau L}\Psi_{-\tau}(U^n)+e^{\tau L}\mathcal R_{-\tau}^n.
\end{align}
Hence we obtain
\begin{equation}
    U^{n+1}=e^{2\tau L}U^{n-1}+[\Psi_{\tau}(U^n)-e^{2\tau L}\Psi_{-\tau}(U^n)]+[\mathcal R_{\tau}^n-e^{2\tau L}\mathcal R_{-\tau}^n].
\end{equation}
This inspires us to define the following two-step scheme:
\begin{equation}\label{general scheme}
    \begin{aligned}
      U^{n+1}&=e^{2\tau L}U^{n-1}+\Psi_{\tau}(U^n)-e^{2\tau L}\Psi_{-\tau}(U^n),\\
      U^1:&=e^{\tau L}U^0+\Psi_{\tau}(U^0).
    \end{aligned}
\end{equation}

It can be seen that the numerical scheme constructed using the above method exhibits symmetry. Let $\tau$ be its oppsite, we have
\begin{equation}\label{oppsite}
    U^{n-1}=e^{-2\tau L}U^{n+1}+\Psi_{-\tau}(U^n)-e^{-2\tau L}\Psi_{\tau}(U^n).
\end{equation} 
After multiplying both sides of (\ref{oppsite}) by $e^{2\tau L}$, it is clear to verify that the scheme (\ref{general scheme}) is symmetric.

If we set the higher-order term (corresponding to the classic Lie splitting method) to be 
\[
\Psi_{\tau}(U^n)=\tau e^{\tau L}F(U^n),
\]
we obtain the first-order symmetric scheme definitiond in (\ref{sLRI1}). As a second application of the above procedure, we shall use the Lie splitting method with low-regularity correction terms in \cite{Li2023} and (\ref{general scheme}) to construct the following second-order symmetric scheme.

The low-regularity correction of Lie splitting definitiond in \cite{Li2023} is
\begin{equation}\label{LBY LR2}
    U^{n+1}=e^{\tau L}U^n+\tau e^{\tau L}F(U^n)+\tau^2e^{\tau L}\varphi_2(-2\tau L)H(U^n),
\end{equation}
where
\begin{equation}
    \varphi_2(A):=A^{-2}(e^{A}-A-I)
\end{equation}
and 
\begin{equation}
    H(U(t_n)):=(-f(u(t_n)),f^{\prime}(u(t_n))\partial_tu(t_n))^T.
\end{equation}
We substitute 
\begin{equation}
    \Psi_{\tau}(U^n)=\tau e^{\tau L}F(U^n)+\tau^2e^{\tau L}\varphi_2(-2\tau L)H(U^n)
\end{equation}
into (\ref{general scheme}), yielding a symmetric scheme (denoted by sLRI2)
\begin{equation}\label{sLRI2}
    \begin{aligned}
        U^{n+1}=&e^{2\tau L}U^{n-1}+2\tau e^{\tau L}F(U_n)+\tau^2e^{\tau L}[\varphi_2(-2\tau L)-\varphi_2(2\tau L)]H(U^n)\\
    =&e^{2\tau L}U^{n-1}+2\tau e^{\tau L}F(U_n)\\
    &\qquad\qquad\ +\underbrace{(2L)^{-1}[2\tau e^{2\tau L}+(2L)^{-1}(e^{-\tau L}-e^{3\tau L})]}_{\tau^2e^{\tau L}\phi(-2\tau L)}H(U^n),\quad n\ge1,\\
        U^1=&e^{\tau L}U^0+\tau e^{\tau L}F(U^0)+\tau^2e^{\tau L}\varphi_2(-2\tau L)H(U^0),
    \end{aligned}
\end{equation}
where
\begin{equation}
    \phi(A)=2A^{-2}(\sinh A-A).
\end{equation}

\section{Convergence results}
In this section, we present the convergence results for the two symmetric low-regularity integrators constructed in Section~\ref{sec:construction}. We place particular emphasis on the first-order scheme (\ref{sLRI1}), for which symmetry yields an improved order of convergence. 

By Stone's theorem and the skew-adjointness of the operator $L$ (cf. \cite{Pazy1983}), $iL$ generates a unitary group on $H_0^1(\mathbb T)\times L^2(\mathbb T)$, equivalently $\{e^{tL}\}_{t\in\br}$ is a $C_0$ group on  $H_0^1(\mathbb T)\times L^2(\mathbb T)$. Then, by the contraction mapping principle and the global Lipschitz property (\ref{Lip}) of $f$, the system (\ref{U_t-LU=F(U)}) has a unique energy solution $U\in L^{\infty}(0,T;H_0^1(\mathbb T)\times L^2(\mathbb T))$ satisfying the the variation-of-constants formula given in (\ref{vcf}).

The energy norm of Klein-Gordon equation (\ref{U_t-LU=F(U)}) is 
\begin{equation}
    |(w_1,w_2)^T|_1=|W|_1:=\left(||\nabla w_1||_{L^2(\Bbb T)}^2+||w_2||_{L^2(\Bbb T)}^2\right)^{\frac12}.
\end{equation}
More generally, we denote
\begin{equation}\label{energynorm}
    |W|_{\alpha}:=\left(||w_1||_{\dot{H}^{\alpha}(\Bbb T)}^2+||w_2||_{H^{\alpha-1}(\Bbb T)}^2\right)^{\frac12},\ \alpha\in[0,2].
\end{equation}
And some non-energy norms are useful for error analysis:
\begin{equation}\label{nonenergynorm}
        ||W||_{\alpha}:=\left(||w_1||_{H^{\alpha}(\Bbb T)}^2+||w_2||_{H^{\alpha-1}(\Bbb T)}^2\right)^{\frac12},\ \alpha\in[0,2].
\end{equation}
When $\alpha<1$, the norms in (\ref{energynorm}) and (\ref{nonenergynorm}) corresponding to the second component are taken to be the Sobolev norm of the dual space of $H^{1-\alpha}(\Bbb T)$.

It is known (see \cite{Kato1967,EngelNagel2000}) that the $C_0$ group $\{e^{tL}\}_{t\in\br}$ is isometric and bounded in the norm $|\cdot|_{\alpha}$ and $||\cdot||_{\alpha}$ respectively. That is, for any $W\in [H^{\alpha}(\Bbb T)\cap H_0^1(\Bbb T)]\times H^{\alpha-1}(\Bbb T)$ and $\alpha\in[1,2]$, we have
\begin{equation}
        |e^{tL}W|_{\alpha}=|W|_{\alpha}\ ,\ 
        ||e^{tL}W||_{\alpha}\lesssim||W||_{\alpha}\ ,\ \forall t\in\br.
\end{equation}
Furthermore, the following estimate for $F(U)$ defined in (\ref{U,U^0,F(U)}) will be used frequently in the sequel: for any $\alpha\in[1,2]$ and $U=(u,v)\in H^{\alpha}(\Bbb T)\times H^{\alpha-1}(\Bbb T)$, we have 
\begin{equation}
        ||F(U)||_{\alpha}=||f(u)||_{H^{\alpha-1}(\Bbb T)}\lesssim ||U||_{\alpha-1}+1.
\end{equation}

We write the multi-steps scheme (\ref{general scheme}) in a matrix form
\begin{equation}
    \mathbf{U}^{n+1}=\bm{\Phi}_\tau(\mathbf{U}^n),
\end{equation}
where $\mathbf{U}^n=(U^n,U^{n-1})^T$, $n=1,2,\cdots$ and 
\begin{equation}\label{matrix form}
    \bm{\Phi}_\tau(\mathbf{U}^n):=\begin{pmatrix}e^{2\tau L}U^{n-1}+\Psi_\tau(U^n)-e^{2\tau L}\Psi_{-\tau}(U^n)\\U^n\end{pmatrix}.
\end{equation}
To estimate the stability of the (\ref{general scheme}), we define a product energy norm on $ [H^{\alpha}(\Bbb T)\times H^{\alpha-1}(\Bbb T)]^2$ as 
\begin{equation}\label{product norm}
    |||\mathbf{V}|||_{\alpha}:=\sqrt{||[\mathbf{V}]_1||_{\alpha}^2+||[\mathbf{V}]_2||_{\alpha}^2},\quad \alpha\in [1,2],
\end{equation}
where $[v]_j$ is the $j$-th component of vector $v$.
In the following, we shall frequently use the fractional Hölder continuity (cf. \cite{Bahouri2011}):
\begin{lemma}\label{exp-I}
    Suppose that $r\in(0,1)$ and $W\in H^{1+r}(\Bbb T)\times H^{r}(\Bbb T)$. Then it holds that
    \begin{equation}
        ||(e^{tL}-I)W||_1\lesssim t^{r}||W||_{1+r}.
    \end{equation}
\end{lemma}
And the bilinear estimates (cf. \cite{KatoPonce1988}) is also useful in the sequel:
\begin{lemma}[Kato-Ponce]
    For $\alpha>\frac d2$ and any $f,g\in H^{\alpha}(\Bbb T)$, we have
    \begin{equation}
        ||fg||_{H^{\alpha}(\Bbb T)}\lesssim ||f||_{H^{\alpha}(\Bbb T)}||g||_{H^{\alpha}(\Bbb T)},
    \end{equation}
    and
    \begin{equation}
        ||fg||_{H^{\beta}(\Bbb T)}\lesssim ||f||_{H^{\alpha}(\Bbb T)}||g||_{H^{\beta}(\Bbb T)},\quad \forall\beta\in[0,\alpha].
    \end{equation}
\end{lemma}

\subsection{Convergence analysis for sLRI1}\label{sec:convergence sLRI1}
We first investigate the stability of (\ref{sLRI1}), which will be useful for estimating the global error of sLRI1.
\begin{proposition}
    Suppose the correction term for (\ref{matrix form}) to be
    \[
    \Psi_{\tau}(U^n)=\tau e^{\tau L}F(U^n),
    \]
    and $\mathbf{U}=(U_1,U_2)^T, \mathbf{V}=(V_1,V_2)^T$ are both in $[H^{s}(\Bbb T)\times H^{s-1}(\Bbb T)]^2$ for $d=1,2,3$ and $\max(1,d/2)< s\le2$. Then there exists $\tau_0>0$ sufficiently
    small and a constant $M$ independent to $\tau$ such that for $0<\tau\le\tau_0$,
\begin{equation}\label{stable sLRI1}
        |||\bm{\Phi}_\tau(\mathbf{U})-\bm{\Phi}_\tau(\mathbf{V})|||_s\le (1+\tau M(||U_1||_s,||V_1||_s)) |||\mathbf{U}-\mathbf{V}|||_s.
    \end{equation}
\end{proposition}
\begin{proof}
    Let $[U_1]_1=u_1,[V_1]_1=v_1$. By Theorem 4.5 in \cite[p. 222]{Pazy1983}, we have
    \begin{align*}
         &|||\bm{\Phi}_\tau(\mathbf{U})-\bm{\Phi}_\tau(\mathbf{V})|||_{s}\\
         \le& ||e^{2\tau L}(U_2-V_2)||_s+||U_1-V_1||_s\\
         &+||\Psi_{\tau}(U_1)-\Psi_{\tau}(V_1)||_s+||\Psi_{-\tau}(U_1)-\Psi_{-\tau}(V_1)||_s\\
         \le &(1+2\tau)|||\mathbf{U}-\mathbf{V}|||_s+2\tau||F(U_1)-F(V_1)||_s.\\
    \end{align*} 
    Letting $\mu(u,v):=\int_0^1f'(v+\tau (u-v))\dd\tau$, since $s>d/2$, the product estimate follows
\begin{align}\label{F(U)-F(V)}
        ||F(U_1)-F(V_1)||_s=&||f(u_1)-f(v_1)||_{H^{s-1}(\mathbb T )}\nonumber\\
        = &||\mu(u_1,v_1)(u_1-v_1)||_{H^{s-1}(\mathbb T )}\nonumber\\
        \le & C(s,d)||\mu(u_1,v_1)||_{H^{s}(\mathbb T )}||u_1-v_1||_{H^{s-1}(\mathbb T )}\nonumber\\
        \le&C(L,s,d)(1+||v_1+\tau(u_1-v_1)||_{H^{s}(\mathbb T )})||u_1-v_1||_{H^{s-1}(\mathbb T )}\nonumber\\
        \le &C(L,s,d)(1+||u_1||_{H^{s-1}(\mathbb T )}+||v_1||_{H^{s-1}(\mathbb T )})|||\mathbf{U}-\mathbf{V}|||_s\nonumber\\
        :=&M(||u_1||_{H^{s}(\mathbb T )},||v_1||_{H^{s}(\mathbb T )})|||\mathbf{U}-\mathbf{V}|||_s.
\end{align}
Based on the above analysis, the result \eqref{stable sLRI1} is obtained immediately. 
\end{proof}
\begin{remark}\label{stable proof}
From the above proof, it is clear that, regardless of the choice of the higher-order term $\Psi_{\tau}$, one always has
\begin{equation}
    |||\bm{\Phi}_\tau(\mathbf{U})-\bm{\Phi}_\tau(\mathbf{V})|||_{s}\le ||e^{2\tau L}||\cdot |||\mathbf{U}-\mathbf{V}|||_s+2||\Psi_{\tau}(U_1)-\Psi_{\tau}(V_1)||_s.
\end{equation}
In other words, as long as the following conditions are satisfied:
(1). the operator norm $||e^{2\tau L}||$ admits an upper bound of the form $1+C\tau$;
(2). the given scheme $U^{n+1}=e^{\tau L}U^n+\Psi_{\tau}(U^n)$ is stable;
then the symmetric scheme defined by (\ref{matrix form}) is stable with respect to the norm (\ref{product norm}).
\end{remark}
Next we have the following improved global convergence of (\ref{sLRI1}).
\begin{theorem}\label{converge sLRI1}
    Letting $0\le r<1$ and $(u^0,v^0)\in H^{1+r}(\Bbb T)\times H^{r}(\Bbb T)$, the numerical solution $U^n$ given by (\ref{sLRI1}) satisfies the error bound
    \begin{equation}
        ||U(t_n)-U^n||_1\le C_1(T,r)\tau^{1+r},\quad 0\le n\le T/\tau,
    \end{equation}
    where the time step size $\tau$ is small sufficiently, and $C_1$ is some positive constant independent of the stepsize $\tau$ but may depend on $T$ and $r$.
\end{theorem}
\begin{proof}
Set
\begin{equation}\label{err R}
    \begin{aligned}
        \mathcal R_{\tau}(t_n):&=U(t_{n+1})-e^{2\tau L}U(t_{n-1})-2\tau e^{\tau L}F(U(t_n))\\
        &=e^{\tau L}(R_1(t_n)+R_2(t_n))
    \end{aligned}
\end{equation}
with
\begin{align}
    R_1(t_n):=&\int_{-\tau}^{\tau}e^{-sL}[F(U(t_n+s))-F(e^{sL}U(t_n))]\dd s,\\
    R_2(t_n):=&\int_{-\tau}^{\tau}e^{-sL}[F(e^{sL}U(t_n))-F(U(t_n))]\dd s.
\end{align}

To estimate $R_1(t_n)$, we define the integrand by
\begin{equation}
    X(\sigma,s):=F(e^{(s-\sigma)L}U(t_n+\sigma)),\quad 0\le|\sigma|\le|s|\le\tau
\end{equation}
and
\begin{equation}
    \partial_{\sigma}X(\sigma,s):=X_1(\sigma,s).
\end{equation}
Then we have
\begin{align}\label{R_1}
    R_1(t_n)=&\int_{-\tau}^{\tau}e^{-sL}(X(s,s)-X(0,s))\dd s\nonumber\\
    =&\int_{-\tau}^{\tau}e^{-sL}\int_{0}^{s}\partial_{\sigma}X(\sigma,s)\dd\sigma\dd s\nonumber\\
    =&\int_{0}^{\tau}e^{-sL}\int_{0}^{s}X_1(\sigma,s)-\int_{-\tau}^0e^{-sL}\int_s^0X_1(-\sigma,-s)\dd\sigma\dd s\nonumber\\
    =&\int_0^{\tau}\int_0^se^{-sL}X_1(\sigma,s)-e^{sL}X_1(-\sigma,-s)\dd\sigma\dd s.
\end{align}

For simplicity of notation, we define the following conventions.
\begin{equation}\label{U,V,W}
    \begin{aligned}
        U(t_n+t):=&U_n(t)=(u_n(t),v_n(t))^T,\\
    e^{(s-\sigma)L}U_n(\sigma):=&(\widetilde{u_n}(s;\sigma),\widetilde{v_n}(s;\sigma))^T,\\
    V_n(t):=&e^{-tL}U_n(t),\\ W_n(t):=&e^{-t L}F(e^{t L}V_n(t)).
    \end{aligned}
\end{equation}
By the definition of $V_n$, combining (\ref{vcf}) and (\ref{V_n(t)-V_n(-t)}), we get
\begin{equation}\label{V_n}
    ||V_n(t)-V_n(-t)||_1\lesssim\int_{-t}^{t}||F(U_n(s))||_1\dd s\lesssim t\left(\max_{s\in[-t,t]}||U_n(s)||_0\right).
\end{equation}
Now we determine the explicit expression for $X_1$:
\begin{equation}
    X_1(\sigma,s)=DF\bigl(e^{(s-\sigma)L}U_n(\sigma)\bigr)
    \cdot\partial_{\sigma}\!\left(e^{(s-\sigma)L}U_n(\sigma)\right),
\end{equation}
where \(DF(x)\) denotes the Jacobian of \(F\) at $x$. Using the semigroup property
\[
-\partial_{\sigma}\bigl(e^{(s-\sigma)L}\bigr)
   =Le^{(s-\sigma)L}=e^{(s-\sigma)L}L,
\]
we have
\begin{align}\label{X_1}
        X_1(\sigma,s)= &DF\bigl(e^{(s-\sigma)L}U_n(\sigma)\bigr)\left[-Le^{(s-\sigma)L}U_n(\sigma) + e^{(s-\sigma)L} U'_n(\sigma)\right]\nonumber\\
        =& DF\bigl(e^{(s-\sigma)L}U_n(\sigma)\bigr)
        \cdot e^{(s-\sigma)L}\!\left(U'_n(\sigma)-L U_n(\sigma)\right)\nonumber\\
        =&DF\bigl(e^{sL}V_n(\sigma)\bigr)\cdot e^{(s-\sigma)L}F(U_n(\sigma))\nonumber\\
        =&\begin{pmatrix}0\\f^{\prime}(\widetilde{u_n}(s;\sigma))\cdot [e^{sL}W_n(\sigma)]_1\end{pmatrix}.
\end{align}

Therefore, we can rewrite the integrand as
\begin{equation}
    \begin{aligned}
        &e^{-sL}X_1(\sigma,s)-e^{sL}X_1(-\sigma,-s)\\
        =&\underbrace{(e^{-sL}-e^{sL})X_1(\sigma,s)}_{e_{1}(\sigma,s)}+X_1(\sigma,s)-X_1(-\sigma,-s)\\
        =&e_{1}(\sigma,s)+\underbrace{DF(e^{sL}V_n(\sigma))-DF(e^{-sL}V_n(-\sigma)))e^{sL}W_n(\sigma)}_{e_2(\sigma,s)}\\
        &\qquad\quad\,+\underbrace{DF(e^{-sL}V_n(-\sigma))(e^{sL}W_n(\sigma)-e^{-sL}W_n(-\sigma))}_{e_3(\sigma,s)}\\
        :=&e_{1}(\sigma,s)+e_{2}(\sigma,s)+e_{3}(\sigma,s).
    \end{aligned}
\end{equation}

According to the (\ref{exp-I}), (\ref{U,V,W}) and (\ref{X_1}), we obtain
\begin{align}\label{e_1}
       ||e_1(\sigma,s)||_1=&||(e^{-sL}-e^{sL})X_1(\sigma,s)||_1\nonumber\\
       \lesssim & s^r||f^{\prime}(\widetilde{u_n}(s;\sigma))\cdot [e^{sL}W_n(\sigma)]_1||_{H^{r}(\mathbb T )}\nonumber\\
       \lesssim &\tau^r||f^{\prime}(\widetilde{u_n}(s;\sigma))||_{H^{1+r}(\mathbb T )}||e^{sL}W_n(\sigma)||_{1+r}\nonumber\\
       \lesssim &\tau^r(\|f^{\prime}\|_{L^{\infty}(\mathbb T )}+\|f^{\prime\prime}\|_{L^{\infty}(\mathbb T )}\|\widetilde{u_n}(s;\sigma)\|_{H^{1+r}(\mathbb T )})||F(U_n(\sigma))||_{1+r}\nonumber\\
       \lesssim &\tau^r(1+||U_n(\sigma)||_{1})(1+||u_n(\sigma)||_{H^r(\mathbb T )})\nonumber\\
       \lesssim &\tau^r\left(1+\max_{t\in[-\tau,\tau]}||U_n(t)||_{1}+\max_{t\in[-\tau,\tau]}||U_n(t)||_{0}\right).
\end{align}

By the embedding $H^1\hookrightarrow L^4$, we have
\begin{align}\label{e_2}
    ||e_2(\sigma,s)||_1=&||(DF(e^{sL}V_n(\sigma))-DF(e^{-sL}V_n(-\sigma)))e^{sL}W_n(\sigma)||_1\nonumber\\
    = &||[f^{\prime}(\widetilde{u_n}(s;\sigma))-f^{\prime}(\widetilde{u_n}(-s;-\sigma))]\cdot[e^{sL}W_n(\sigma)]_1||_{L^2(\mathbb T )}\nonumber\\
    \le &||f^{\prime}(\widetilde{u_n}(s;\sigma))-f^{\prime}(\widetilde{u_n}(-s;-\sigma))||_{L^4(\mathbb T )}||[e^{sL}W_n(\sigma)]_1||_{L^4(\mathbb T )}\nonumber\\
    \lesssim & ||\widetilde{u_n}(s;\sigma)-\widetilde{u_n}(-s;-\sigma)||_{H^1(\mathbb T )}||e^{sL}W_n(\sigma)||_1\nonumber\\
    \le & ||e^{sL}V_n(\sigma)-e^{-sL}V_n(-\sigma)||_1\cdot||F(U_n(\sigma))||_1\nonumber\\
    \lesssim &(s^r||V_n(-\sigma)||_{1+r}+||V_n(\sigma)-V_n(-\sigma)||_1)(|f(0)|+||u_n(\sigma)||_{L^2(\mathbb T )})\nonumber\\
    \lesssim &\tau^r\left(\max_{s\in[-\tau,\tau]}||U_n(\sigma)||_{1+r}+\max_{s\in[-\tau,\tau]}||U_n(\sigma)||_0\right)\left(1+\max_{t\in[-\tau,\tau]}||U_n(t)||_{0}\right).
\end{align}
The computation in (\ref{V_n}) and (\ref{X_1}) yields
\begin{align}\label{e_3}
    ||e_3(\sigma,s)||_1=&||DF(e^{-sL}V_n(-\sigma))(e^{sL}W_n(\sigma)-e^{-sL}W_n(-\sigma))||_1\nonumber\\
    \lesssim &||e^{sL}W_n(\sigma)-e^{-sL}W_n(-\sigma)||_1\nonumber\\
    \le &||(e^{sL}-e^{-sL})W_n(-\sigma)||_1+||W_n(\sigma)-W_n(-\sigma)||_1\nonumber\\
    \lesssim &s^r||W_n(-\sigma)||_{1+r}+||(e^{\sigma L}-e^{-\sigma L})V_n(-\sigma)||_1+||V_n(\sigma)-V_n(-\sigma)||_1\nonumber\\
    \lesssim &\tau^r||F(U_n(-\sigma))||_{1+r}+\tau^r||U_n(-\sigma)||_{1+r}+\tau\cdot\max_{s\in[-\tau,\tau]}||U_n(s)||_0\nonumber\\
    \lesssim &\tau^r\left(1+\max_{s\in[-\tau,\tau]}||U_n(s)||_{1+r}+\max_{s\in[-\tau,\tau]}||U_n(s)||_0\right).
\end{align}
Substituting (\ref{e_1}), (\ref{e_2}) and (\ref{e_3}) into (\ref{R_1}) follows
\begin{equation}\label{err R_1}
    ||R_1(t_n)||_1\lesssim \tau^{2+r}\left(1+\max_{s\in[-\tau,\tau]}||U_n(s)||_{1+r}+\max_{s\in[-\tau,\tau]}||U_n(s)||_0\right)^2.
\end{equation}

We now turn to $R_2(t_n)$. Analogously to the case of $R_1(t_n)$, define the integrand by 
\[
Y(\sigma):=e^{-\sigma L}F(e^{\sigma L}U(t_{n})),\quad 0\le|\sigma|\le\tau.
\]
Then, it is obtained that
\begin{equation}\label{R_2}
    R_2(t_n)=e^{\tau L}\int_{-\tau}^{\tau}\int_0^sY'(\sigma)\dd\sigma\dd s=e^{\tau L}\int_{0}^{\tau}\int_0^sY'(\sigma)-Y'(-\sigma)\dd\sigma\dd s.
\end{equation}
Computing directly gives 
\begin{align}
    Y'(\sigma)-Y'(-\sigma)=&\left.e^{-\sigma L}\left(\begin{array}{c}-f(\tilde{u}(t_n+\sigma))\\f^{\prime}(\tilde{u}(t_n+\sigma))\tilde{v}(t_n+\sigma)\end{array}\right.\right)-\left.e^{\sigma L}\left(\begin{array}{c}-f(\tilde{u}(t_n-\sigma))\\f^{\prime}(\tilde{u}(t_n-\sigma))\tilde{v}(t_n-\sigma)\end{array}\right.\right)\nonumber\\
    =&e^{-\sigma L}\underbrace{\begin{pmatrix}
                   f(\tilde{u}(t_n-\sigma))-f(\tilde{u}(t_n+\sigma))\\
                   f^{\prime}(\tilde{u}(t_n+\sigma))\tilde{v}(t_n+\sigma)-f^{\prime}(\tilde{u}(t_n-\sigma))\tilde{v}(t_n-\sigma)
                   \end{pmatrix}}_{Y_1(t_n,\sigma)}\nonumber\\
    &+\underbrace{(e^{-\sigma L}-e^{\sigma L})\begin{pmatrix}
                   -f(\tilde{u}(t_n-\sigma))\\
                   f^{\prime}(\tilde{u}(t_n-\sigma))\tilde{v}(t_n-\sigma)
                   \end{pmatrix}}_{Y_2(t_n,\sigma)}\nonumber\\
    :=&e^{-\sigma L}Y_1(t_n,\sigma)+Y_2(t_n+\sigma).
\end{align}
The following estimate holds for $Y_1(t_n,\sigma)$:
\begin{align}\label{Y_1}
||Y_1(t_n,\sigma)||_1\lesssim&||f(\tilde{u}(t_n-\sigma))-f(\tilde{u}(t_n+\sigma))||_{H^1(\mathbb T )}\nonumber\\
    &+||f^{\prime}(\tilde{u}(t_n+\sigma))\tilde{v}(t_n+\sigma)-f^{\prime}(\tilde{u}(t_n-\sigma))\tilde{v}(t_n-\sigma)||_{L^2(\mathbb T )}\nonumber\\
    \lesssim & ||f'(\tilde{u}(t_n+\sigma))\nabla\tilde{u}(t_n+\sigma)-f'(\tilde{u}(t_n-\sigma))\nabla\tilde{u}(t_n-\sigma)||_{L^2(\mathbb T )}\nonumber\\
    &+||\tilde{u}(t_n+\sigma)-\tilde{u}(t_n-\sigma)||_{L^2(\mathbb T )}+||\tilde{v}(t_n+\sigma)-\tilde{v}(t_n-\sigma)||_{L^2(\mathbb T )}\nonumber\\
    \lesssim&||\nabla\tilde{u}(t_n+\sigma)-\nabla\tilde{u}(t_n-\sigma)||_{L^2(\mathbb T )}+||\tilde{v}(t_n+\sigma)-\tilde{v}(t_n-\sigma)||_{L^2(\mathbb T )}\nonumber\\
    &+(1+||\nabla\tilde{u}(t_n+\sigma)||_{L^2})||\tilde{u}(t_n+\sigma)-\tilde{u}(t_n-\sigma)||_{L^2(\mathbb T )}\nonumber\\
    \lesssim &(1+||U(t_n)||_1)||(e^{\sigma L}-e^{-\sigma L})U(t_n)||_1\nonumber\\
    \lesssim &\tau^r(1+||U(t_n)||_{1+r})^2.
\end{align}
For $Y_2(t_n,\sigma)$, we have
\begin{align}\label{Y_2}
    ||Y_2(t_n,\sigma)||_1\lesssim&\tau^r(||f(\tilde{u}(t_n-\sigma))||_{H^{1+r}(\mathbb T )}+||\tilde{u}(t_n-\sigma)||_{H^r(\mathbb T )})\nonumber\\
    \lesssim & \tau^r(1+||f(\tilde{u}(t_n-\sigma))-f(0)||_{H^{1+r}(\mathbb T )}+||\tilde{u}(t_n-\sigma)||_{H^r(\mathbb T )})\nonumber\\
    \lesssim &\tau^r(1+||\nabla\tilde{u}(t_n-\sigma)||_{H^r(\mathbb T )}+2||\tilde{u}(t_n-\sigma)||_{H^r(\mathbb T )})\nonumber\\
    \lesssim &\tau^r(1+2||U(t_n)||_{1+r}.
\end{align}
Therefore combining (\ref{Y_1}), (\ref{Y_2}) and (\ref{R_2}), we obtain the estimation
\begin{align}\label{err R_2}
    ||R_2(t_n)||_1\lesssim &\tau^{2}(||Y_1(t_n,\sigma)||_1+||Y_2(t_n,\sigma)||_1)\nonumber\\
    \lesssim &\tau^{2+r}\left(1+||U(t_n)||_{1+r}\right)^2.
\end{align}

Let $\mathbf{U}(t_{n})=(U(t_{n}),U(t_{n-1}))^T$. Plugging (\ref{err R_1}) and (\ref{err R_2}) into (\ref{err R}) gives the local truncation error estimate
\begin{equation}\label{local err}
        |||\mathbf{U}(t_{n+1})-\bm{\Phi}_\tau(\mathbf{U}^n)|||_{1}
        =||\mathcal R_{\tau}(t_n)-e^{2\tau L}\mathcal R_{-\tau}(t_n)||_1
        \le C_1\left(\sup_t||U(\cdot,t)||_{1+r}\right)\tau^{2+r}.
\end{equation}
For the global error, we first obtain the following estimate from (\ref{stable sLRI1}) and (\ref{local err}):
\begin{align}\label{global err}
    ||U(t_n)-U^n||_1\le &|||\mathbf{U}(t_{n})-\mathbf{U}^n|||_{1}\nonumber\\
    \le &|||\bm{\Phi}_\tau(\mathbf{U}(t_n))-\bm{\Phi}_\tau(\mathbf{U}^n)|||_{1}+|||\mathbf{U}(t_{n+1})-\bm{\Phi}_\tau(\mathbf{U}^n)|||_{1}\nonumber\\
    \le &(1+\tau M(||U^n||_1))|||\mathbf{U}(t_{n})-\mathbf{U}^n|||_{1}+C_1\tau^{2+r}.
\end{align}
Then, applying the discrete Gronwall inequality yields that the global error is of order $1+r$, where the uniform upper bound for $M(||U^n||_1)$ in (\ref{global err}) can be obtained via induction.

\end{proof}

\begin{remark}
The conclusion in (\ref{converge sLRI1}) requires that the exact solution has $(1+r)$-th order regularity, namely
\[
U\in L^{\infty}(0,T;H^{1+r}(\Bbb T)\times H^{r}(\Bbb T)),\quad r\in[0,1).
\]
This indicates that, whenever the exact solution possesses $r$-th order of regularity, the convergence order of the scheme improves by $r$-th orders, a gain that stems precisely from the symmetry of the scheme (\ref{sLRI2}). It is worth noting that, for the Schrödinger equation (cf. \cite{Feng2024}), an analogous result yields only an improvement of $r/2$-th orders in the convergence order. This discrepancy arises from the different properties of the semigroups generated by the operator $L$ and the operator $i\Delta$.
\end{remark}

\subsection{Convergence analysis for sLRI2}\label{sec:convergence sLRI2}
Based on the analysis given in the above section, 
we can deduce the convergence for sLRI2. 
In a similar way, we first study the stability of (\ref{sLRI2}).
\begin{proposition}\label{prop stable 2}
 Consider the correction term for (\ref{matrix form}) defined as
 \[
 \Psi_{\tau}(U^n)=\tau e^{\tau L}F(U^n)+\tau^2e^{\tau L}\varphi_2(-2\tau L)H(U^n),
 \]
 where $\mathbf{U}=(U_1,U_2)^T, \mathbf{V}=(V_1,V_2)^T \in [H^{s}(\Bbb T)\times H^{s-1}(\Bbb T)]^2$. For $d=1,2,3$ and $\max(1,d/2)< s\le2$, there exists a sufficiently small $\tau_0>0$ and a constant $M$ independent of $\tau$ satisfying, for all $0<\tau\le\tau_0$,
\begin{equation}\label{stable sLRI2}
|||\bm{\Phi}_\tau(\mathbf{U})-\bm{\Phi}_\tau(\mathbf{V})|||_s\le (1+\tau M(||U_1||_s,||V_1||_s)) |||\mathbf{U}-\mathbf{V}|||_s.
    \end{equation}
\end{proposition}
\begin{proof}
    For $j=1,2$, let $[U_1]_j=u_j$ and $[V_1]_j=v_j$. Firstly, we have
    \begin{equation}\label{only need to prove}
        ||\Psi_{\tau}(U_1)-\Psi_{\tau}(V_1)||_s\le \tau||F(U_1)-F(V_1)||_s+||\tau^2 e^{\tau L}\varphi_2(-2\tau L)(H(U_1)-H(V_1))||_s.
    \end{equation}
By the (\ref{F(U)-F(V)}), we obtain 
\[
||F(U_1)-F(V_1)||_s\le M(||U_1||_s,||V_1||_s)|||\mathbf{U}-\mathbf{V}|||_s.
\]
Then, it can be derived that
\begin{align}\label{Psi(U)-Psi(V)}
    &||\tau^2e^{\tau L}\varphi_2(-2\tau L)(H(U_1)-H(V_1))||_s\nonumber\\
    =&\Bigg|\Bigg|\int_0^t e^{-2sL}(t-s)(H(U_1)-H(V_1))\dd s\Bigg|\Bigg|_s\nonumber\\
    \le&\tau^2\left(||f(u_1)-f(v_1)||_{H^{s}(\mathbb T )}+||f'(u_1)u_2-f'(v_1)v_2||_{ H^{s-1}(\mathbb T )}\right)\nonumber\\
    \lesssim &\tau^2(||\mu(u_1,v_1)||_{H^{s}(\mathbb T )}||u_1-v_1||_{H^{s}(\mathbb T )})\nonumber\\
    &+\tau^2\left(||f'(u_1)(u_2-v_2)||_{H^{s-1}(\mathbb T )}+||(f'(u_1)-f'(v_1))v_2||_{ H^{s-1}(\mathbb T )}\right)\nonumber\\
    \lesssim &\tau^2(1+||u_1||_{H^{s}(\mathbb T )}+||v_1||_{H^{s}(\mathbb T )})|||\mathbf{U}-\mathbf{V}|||_s\nonumber\\
    &+ \tau^2(||f'(u_1)||_{ H^{s-1}(\mathbb T )}|||\mathbf{U}-\mathbf{V}|||_s+||f'(u_1)-f'(v_1)||_{H^{s-1}(\mathbb T )}||v_2||_{H^{s}(\mathbb T )})\nonumber\\
    \le & \tau^2C(||U_1||_s,||V_1||_s)|||\mathbf{U}-\mathbf{V}|||_s.
\end{align}
By the remark \ref{stable proof}, combining the (\ref{F(U)-F(V)}), (\ref{only need to prove}) and (\ref{Psi(U)-Psi(V)}), we proved the stability \eqref{stable sLRI2} of (\ref{sLRI2}). 
\end{proof}

Since that in this case the scheme (\ref{LBY LR2}) is of even order, we cannot expect any improvement in the convergence order by symmetrisation. The local error of (\ref{LBY LR2}) was studied in \cite{Li2023}, combining the \cite[Section 2.2]{Li2023} and (\ref{stable sLRI2}) we obtain the global error bound of sLRI2 (\ref{sLRI2}).

\begin{theorem}\label{converge sLRI2}
    For $d=1,2,3$ and $(u^0,v^0)\in H^s(\Bbb T)\times H^{s-1}(\Bbb T)$ with $1+d/4<s\le 2$, the numerical solution $U^n$ given by (\ref{sLRI1}) satisfies the error bound
    \begin{equation}
        ||U(t_n)-U^n||_1\le C_2(T,r)\tau^{2},\quad 0\le n\le T/\tau,
    \end{equation}
    where the time step size $\tau$ is sufficiently  small and $C_2$ is a positive constant independent of the step size $\tau$ but may depend on $T$ and $r$.
\end{theorem}

\begin{proof}
    Firstly, the local truncation error estimate studied in \cite[Theorem 1.1]{Li2023} yields 
    \begin{equation*}
        |||\mathbf{U}(t_{n+1})-\bm{\Phi}_\tau(\mathbf{U}^n)|||_{1}\le \tilde{C_2}\tau^{3},
    \end{equation*}
    where $\tilde{C_2}$  is some positive constant independent of the step size $\tau$. Then we have 
    \begin{equation*}
        |||\bm{\Phi}_\tau(\mathbf{U})-\bm{\Phi}_\tau(\mathbf{V})|||_s\le (1+\tau M(||U_1||_s,||V_1||_s)) |||\mathbf{U}-\mathbf{V}|||_s
    \end{equation*}
    by the Proposition \ref{prop stable 2}. Thus it is deduced that
    \begin{align*}
    ||U(t_n)-U^n||_1\le &|||\mathbf{U}(t_{n})-\mathbf{U}^n|||_{1}\nonumber\\
    \le &|||\bm{\Phi}_\tau(\mathbf{U}(t_n))-\bm{\Phi}_\tau(\mathbf{U}^n)|||_{1}+|||\mathbf{U}(t_{n+1})-\bm{\Phi}_\tau(\mathbf{U}^n)|||_{1}\nonumber\\
    \le &(1+\tau M(||U_1||_s,||V_1||_s))|||\mathbf{U}(t_{n})-\mathbf{U}^n|||_{1}+\tilde{C_2}\tau^{3}.
\end{align*}
Using the discrete Gronwall inequality and induction we obtain
\begin{equation*}
        ||U(t_n)-U^n||_1\le C_2(T,r)\tau^{2},\quad 0\le n\le T/\tau.
    \end{equation*}
\end{proof}

\section{Numerical experiments}\label{sec:Numerical experiments}
In this section, we present numerical experiments for the Klein-Gordon equation (\ref{NKGE}) on the time interval $T=1$ with the nonlinear term $f(u)=\sin(u)$, explicitly, this can be written as
\begin{equation}\label{numerical}
    \begin{cases}\partial_{t t} u-\partial_{xx} u=\sin(u) & \text { in } \bbt \times(0, 1], \\ u(x,0)=u^0(x),\ \partial_t u(x,0)=v^0(x) & \text { in } \bbt.\end{cases}
\end{equation}
The choice of a sine-type nonlinearity is motivated by its widespread appearance in physical models, where equations of this form arise naturally in, for example, relativistic field theories and lattice dynamics. For the temporal discretization, we employ the symmetric low-regularity integrators introduced in this paper, while in space we use a Fourier spectral method with $N=2^{12}$ grid points $x_j=j\pi/N$ with $j=-N,-N+1,\cdots,N-1$. The experiments investigate the convergence behavior, computational efficiency and  long-time preservation of the energy of sLRI1 (\ref{sLRI1}) and sLRI2 (\ref{sLRI2}) under various initial data. Special emphasis is placed on comparing sLRI1 with LRI1 and sLRI2 with LRI2 in order to highlight the differences in accuracy, performance, and long-time stability between the symmetric and nonsymmetric schemes.

The above-mentioned numerical techniques are:
\begin{itemize}
    \item \textbf{LRI1} in \cite{Li2023}:
    \begin{equation*}
        \begin{pmatrix}u_{n+1}\\v_{n+1}\end{pmatrix}=e^{\tau L}\begin{pmatrix}u_{n}\\v_{n}\end{pmatrix}+\tau e^{\tau L}\begin{pmatrix}0\\\sin(u_n)\end{pmatrix}.
    \end{equation*}
    \item \textbf{sLRI1} in (\ref{sLRI1}):
    \begin{equation*}
        \begin{aligned}
            \begin{pmatrix}u_{n+1}\\v_{n+1}\end{pmatrix}=&e^{2\tau L}\begin{pmatrix}u_{n-1}\\v_{n-1}\end{pmatrix}+2\tau e^{\tau L}\begin{pmatrix}0\\\sin(u_n)\end{pmatrix},\\
            \begin{pmatrix}u_{1}\\v_{1}\end{pmatrix}=&e^{\tau L}\begin{pmatrix}u_{0}\\v_{0}\end{pmatrix}+\tau e^{\tau L}\begin{pmatrix}0\\\sin(u_0)\end{pmatrix}.
        \end{aligned}
    \end{equation*}
    \item \textbf{LRI2} in \cite{Li2023}:
    \begin{equation*}
        \begin{aligned}
            \begin{pmatrix}u_{n+1}\\v_{n+1}\end{pmatrix}=&e^{\tau L}\begin{pmatrix}u_{n}\\v_{n}\end{pmatrix}+\tau e^{\tau L}\begin{pmatrix}0\\\sin(u_n)\end{pmatrix}\\&+(2L)^{-1}\left[\tau e^{\tau L}-(2L)^{-1}(e^{\tau L}-e^{-\tau L})\right]\begin{pmatrix}-\sin(u_n)\\\cos(u_n)v_n\end{pmatrix}.
        \end{aligned}
    \end{equation*}
    \item \textbf{sLRI2} in (\ref{sLRI2}):
    \begin{equation*}
        \begin{aligned}
            \begin{pmatrix}u_{n+1}\\v_{n+1}\end{pmatrix}=&e^{2\tau L}\begin{pmatrix}u_{n-1}\\v_{n-1}\end{pmatrix}+2\tau e^{\tau L}\begin{pmatrix}0\\\sin(u_n)\end{pmatrix}\\&+(2L)^{-1}[2\tau e^{\tau L}-(2L)^{-1}(e^{3\tau L}-e^{-\tau L})]\begin{pmatrix}-\sin(u_n)\\\cos(u_n)v_n\end{pmatrix},\\
            \begin{pmatrix}u_{1}\\v_{1}\end{pmatrix}=&e^{\tau L}\begin{pmatrix}u_{0}\\v_{0}\end{pmatrix}+\tau e^{\tau L}\begin{pmatrix}0\\\sin(u_0)\end{pmatrix}.
        \end{aligned}
    \end{equation*}
\end{itemize}

We generate an $H^\theta\times H^{\theta-1}$-initial datum $(u^0,v^0)$ through

\begin{equation*}
    u^0(x)=\frac{\phi_1(x)}{\|\phi_1\|_{H^1}}, \quad \phi_1(x)=\sum_{l=-N}^{N-1} \frac{\xi_l}{\langle l\rangle^{\theta+1 / 2}} e^{i l x}, \quad x \in \mathbb{T},
\end{equation*}
and
\begin{equation*}
    v^0(x)=\frac{\phi_2(x)}{\|\phi_2\|_{L^2}}, \quad \phi_2(x)=\sum_{l=-N}^{N-1} \frac{\zeta_l}{\langle l\rangle^{\theta-1 / 2}} e^{i l x}, \quad x \in \mathbb{T},
\end{equation*}
where $\xi_l$ and $\zeta_l$ are both random numbers uniformly distributed in $[0,1],  N=2^{12}$ being the maximum frequency of the numerical initial datum (which is chosen much larger than the maximum frequency of the numerical solution), and

$$
\langle s\rangle=\left\{\begin{array}{ll}
|s|, & s \neq 0, \\
1, & s=0,
\end{array} \quad s \in \mathbb{R}.\right.
$$

\textbf{Accuracy}.

To estimate the global errors at $T=1$, we introduce the following error function
\begin{equation*}
    \mathrm{err}(T):=\frac{||u_n-u(T)||_{H^1(\Bbb T)}}{||u(T)||_{H^1(\Bbb T)}}+\frac{||v_n-v(T)||_{L^2(\Bbb T)}}{||v(T)||_{L^2(\Bbb T)}}.
\end{equation*}

\begin{figure}[t!]
    \centering
    \begin{minipage}{0.45\textwidth}
        \centering
        \includegraphics[width=\textwidth]{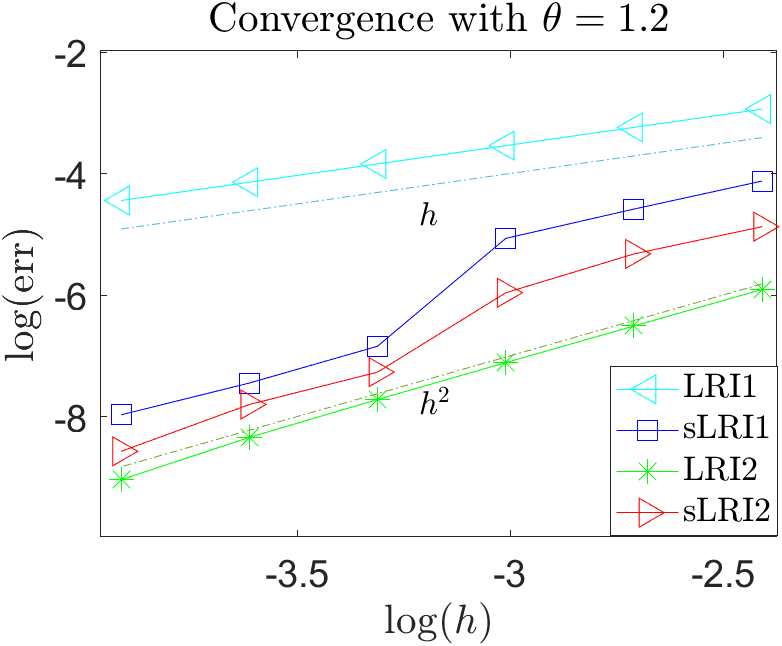}
          
    \end{minipage}
    \hfill
    % 右侧图片，手动添加 
    \begin{minipage}{0.45\textwidth}
        \centering
        \includegraphics[width=\textwidth]{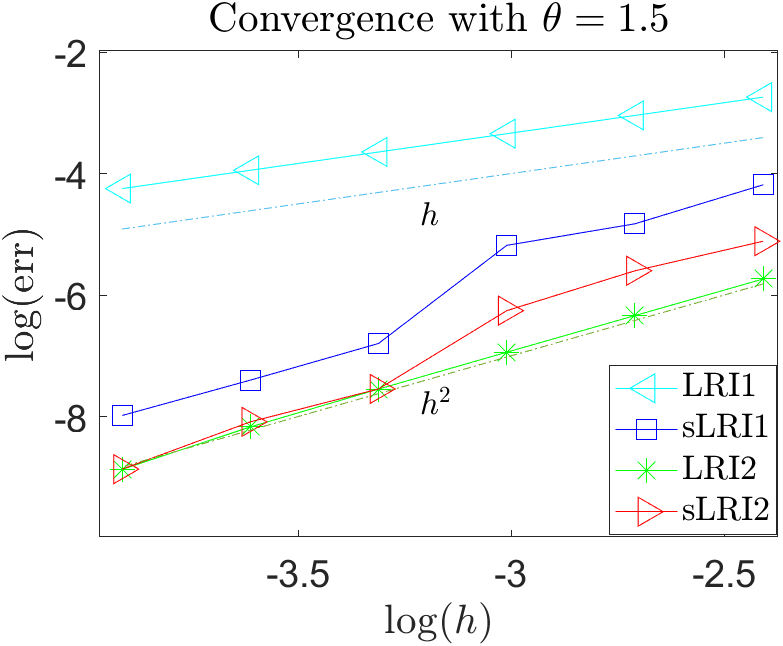}
          
    \end{minipage}
    
    \caption{Global errors $\mathrm{err}(1)$ of the the differential schemes for the NLSE (\ref{numerical}) with initial data in $H^{\theta}(\Bbb T)$ of regularity $\theta=1.0$  and $\theta=1.5$ .}
    \label{1.2 1.5}
\end{figure}
We first present the convergence orders of sLRI1 and sLRI2 under rough initial conditions, i.e. $\theta=1.2$ and $\theta=1.5$. The convergence orders shown in figure \ref{1.2 1.5} reflect the low-regularity requirements of the algorithms. Next, we provide the convergence orders of sLRI1 and sLRI2 under smoother initial conditions ($\theta\ge2$). Figure \ref{2 10} shows that their convergence orders are both $2$, which aligns with the results of Theorem \ref{converge sLRI1} and Theorem \ref{converge sLRI2}. Subsequently, we focus on demonstrating how the convergence order of sLRI1 varies with initial data of different regularities. As can be seen from Figures \ref{1.0 1.2}, \ref{1.4 1.6} and \ref{1.8 2}, as the regularity index $\theta$ increases within the interval [1, 2], the convergence order of sLRI1 indeed improves accordingly, until it reaches order 2 when $\theta\ge2$. This precisely confirms the results of our Theorem \ref{converge sLRI1}.

\begin{figure}[t!]
    \centering
    \begin{minipage}{0.45\textwidth}
        \centering
        \includegraphics[width=\textwidth]{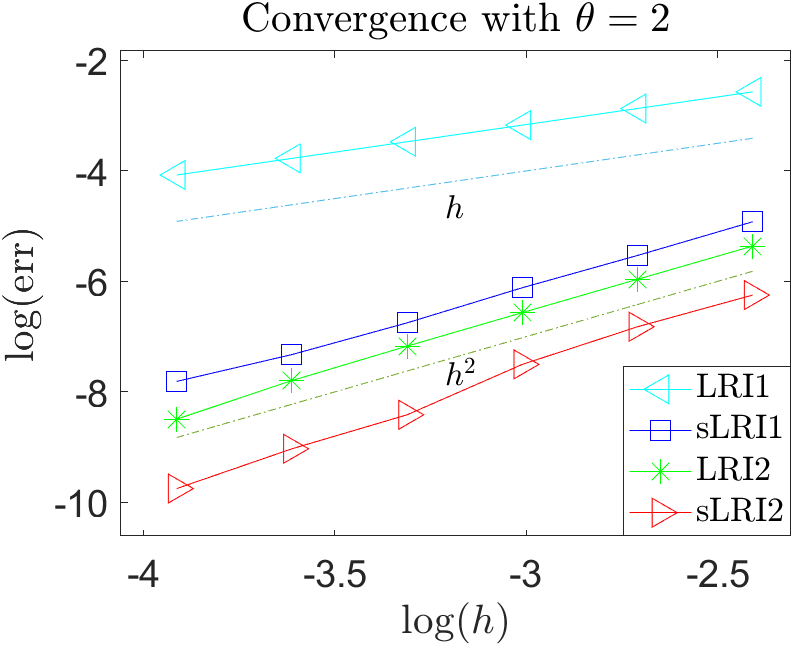}
        
    \end{minipage}
    \hfill
    % 右侧图片，手动添加 
    \begin{minipage}{0.45\textwidth}
        \centering
        \includegraphics[width=\textwidth]{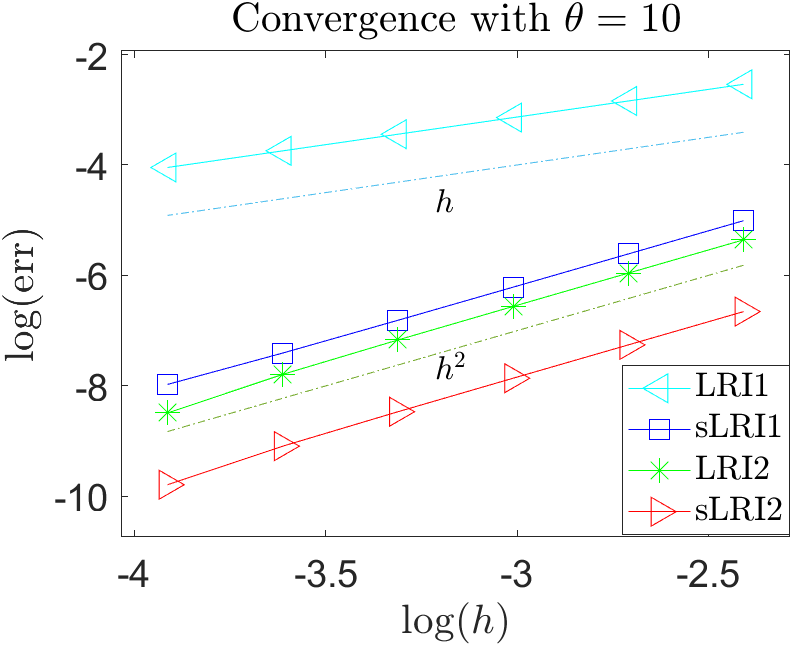}
        
    \end{minipage}
    
    \caption{Global errors $\mathrm{err}(1)$  of the differential schemes for the NLSE (\ref{numerical}) with smooth initial data of regularity $\theta=2$  and $\theta=10$ .}
    \label{2 10}
\end{figure}

\begin{figure}[t!]
    \centering
    \begin{minipage}{0.45\textwidth}
        \centering
        \includegraphics[width=\textwidth]{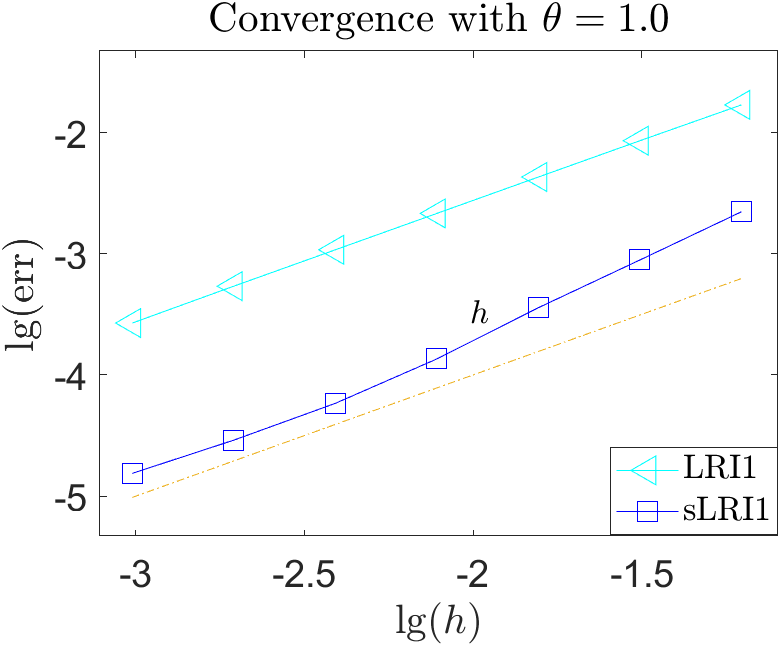}
        
    \end{minipage}
    \hfill
    % 右侧图片，手动添加 
    \begin{minipage}{0.45\textwidth}
        \centering
        \includegraphics[width=\textwidth]{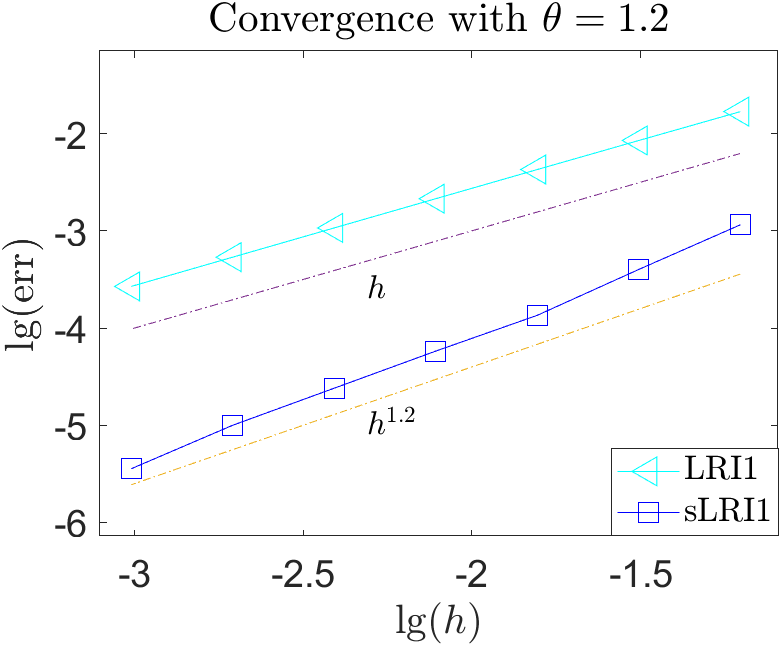}
        
    \end{minipage}
    
    \caption{The convergence orders of sLRI1 under regularity regimes $\theta=1$ and $\theta=1.2$ .}
    \label{1.0 1.2}
\end{figure}

\begin{figure}[t!]
    \centering
    \begin{minipage}{0.45\textwidth}
        \centering
        \includegraphics[width=\textwidth]{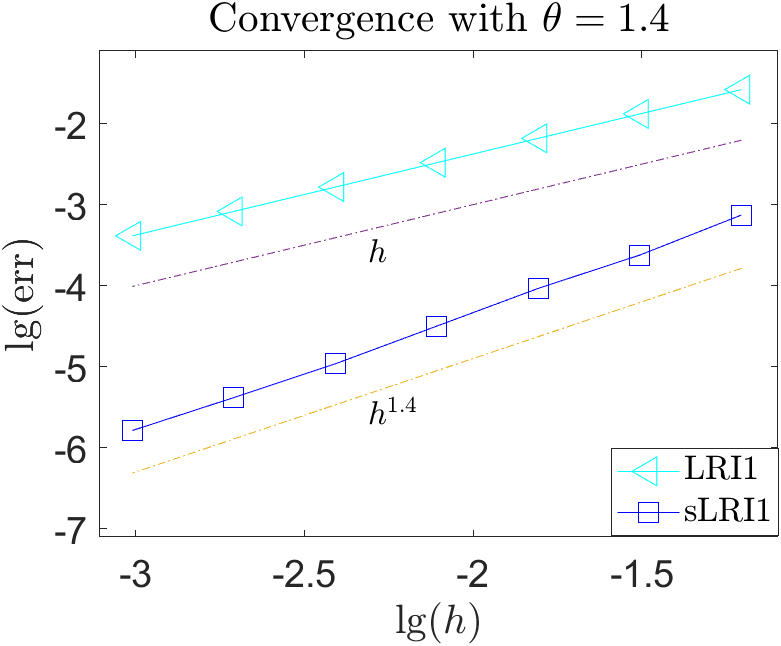}
        
    \end{minipage}
    \hfill
    % 右侧图片，手动添加 
    \begin{minipage}{0.45\textwidth}
        \centering
        \includegraphics[width=\textwidth]{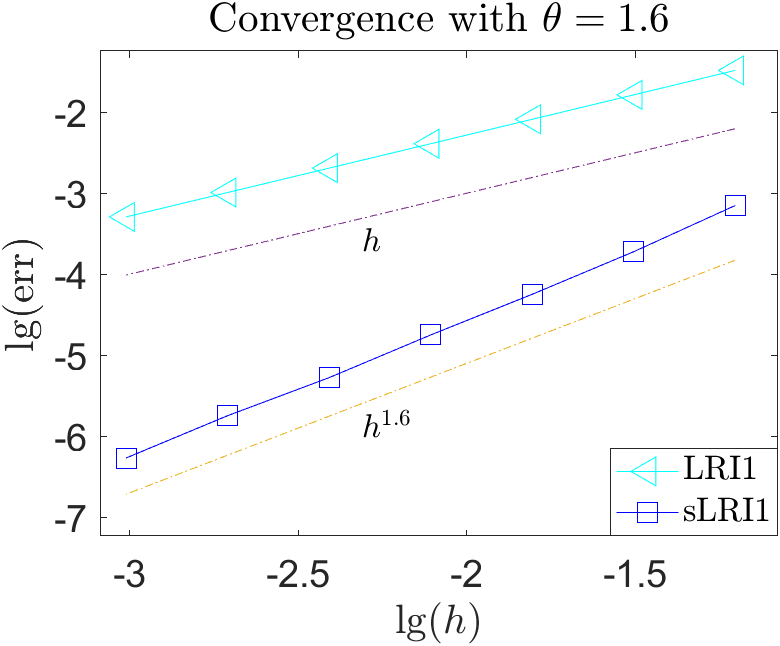}
        
    \end{minipage}
    
    \caption{The convergence orders of sLRI1 under regularity regimes $\theta=1.4$ and $\theta=1.6$ .}
    \label{1.4 1.6}
\end{figure}

\begin{figure}[t!]
    \centering
    \begin{minipage}{0.45\textwidth}
        \centering
        \includegraphics[width=\textwidth]{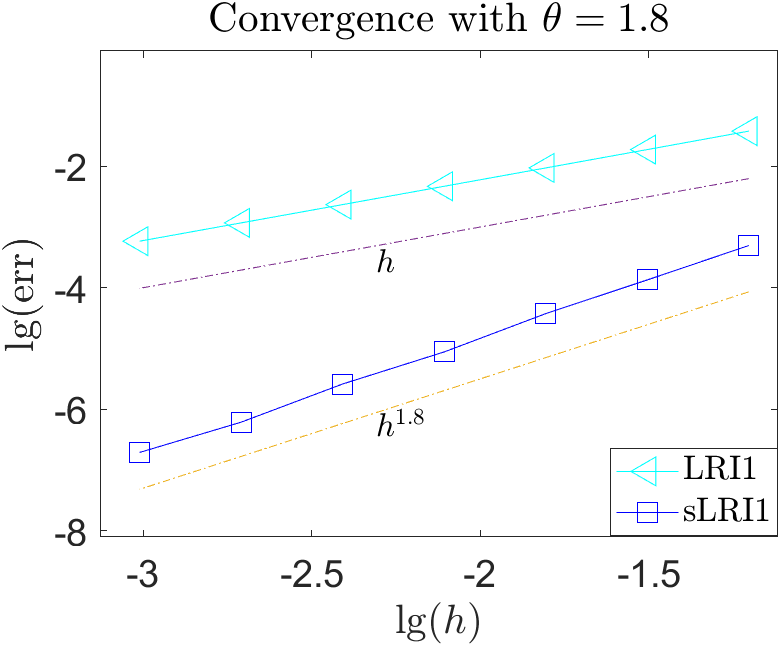}
        
    \end{minipage}
    \hfill
    % 右侧图片，手动添加 
    \begin{minipage}{0.45\textwidth}
        \centering
        \includegraphics[width=\textwidth]{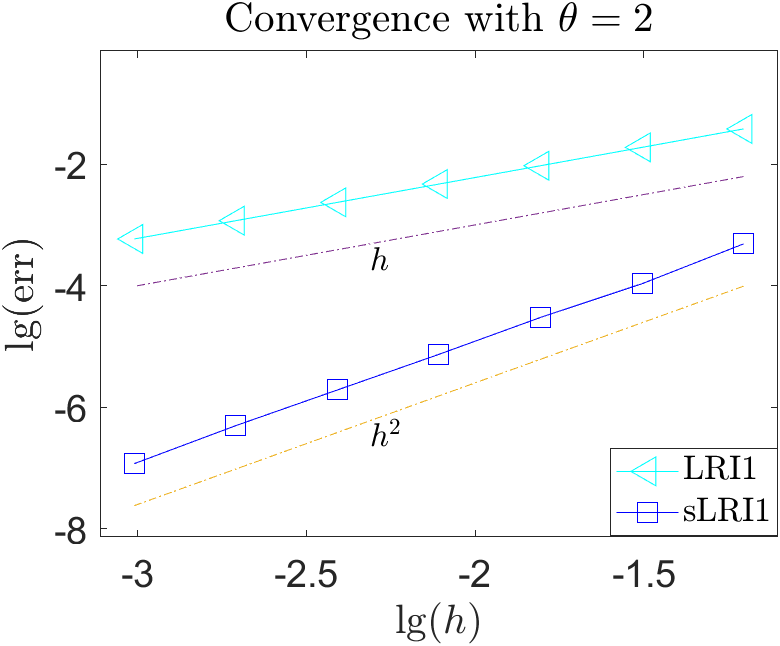}
        
    \end{minipage}
    
    \caption{The convergence orders of sLRI1 under regularity regimes $\theta=1.8$ and $\theta=2$ .}
    \label{1.8 2}
\end{figure}

\textbf{Efficiency}.

The subsequent numerical experiment on CPU time aims to reveal the relationship between computational complexity and actual runtime for the aforementioned single-step scheme and the symmetric two-step method. Figure \ref{CPU 1.2 1.5} and \ref{CPU 1.8 10} gives a plot of Cpu runtime versus error for the above four methods undering Sobolev index $\theta=1.2,1.5,1.8$ and $10$.

As shown, with the exception of LRI1 (\ref{LRI1}), the other three methods demonstrate a significant advantage in computational time, despite some sacrifice in algorithmic simplicity. Furthermore, it can be observed that the two-step method does not perform noticeably worse in runtime than the single-step method. It is due to the cancellation of certain higher-order terms during the construction of the symmetric method.

\begin{figure}[t!]
    \centering
    \begin{minipage}{0.45\textwidth}
        \centering
        \includegraphics[width=\textwidth]{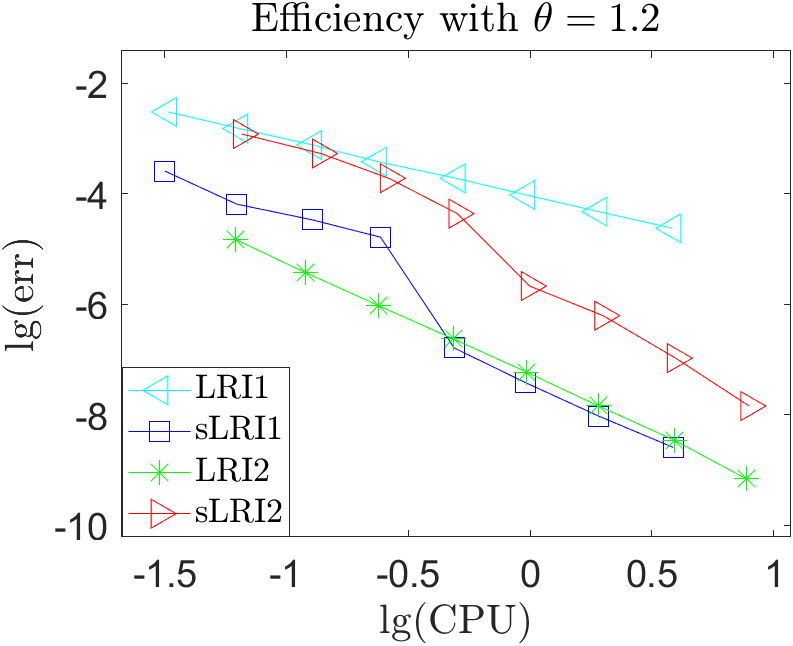}
         
    \end{minipage}
    \hfill
    \begin{minipage}{0.45\textwidth}
        \centering
        \includegraphics[width=\textwidth]{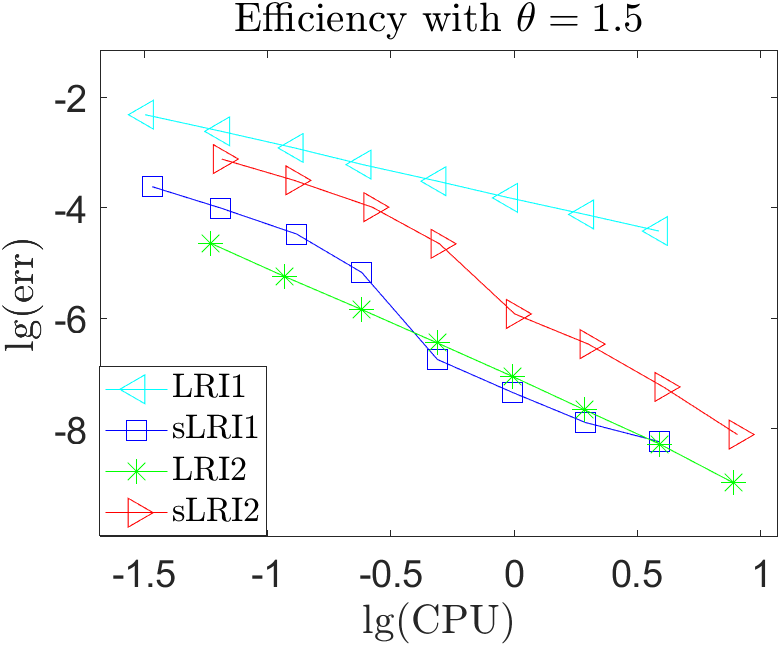}
         
    \end{minipage}
    
    \caption{Comparison of CPU time versus error under initial conditions in $\theta=1.2$  and $\theta=1.5$.}
    \label{CPU 1.2 1.5}
\end{figure}

\begin{figure}[t!]
    \centering
    \begin{minipage}{0.45\textwidth}
        \centering
        \includegraphics[width=\textwidth]{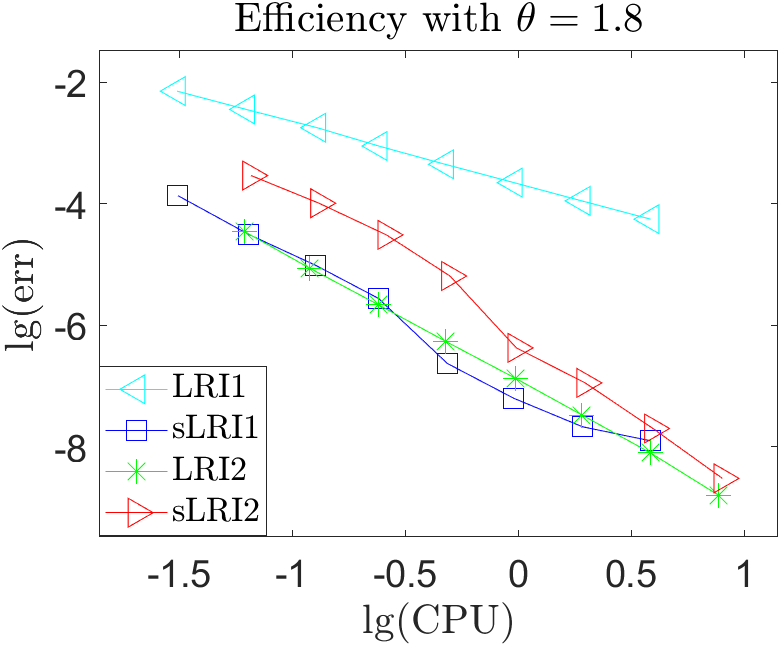}
         
    \end{minipage}
    \hfill
    \begin{minipage}{0.45\textwidth}
        \centering
        \includegraphics[width=\textwidth]{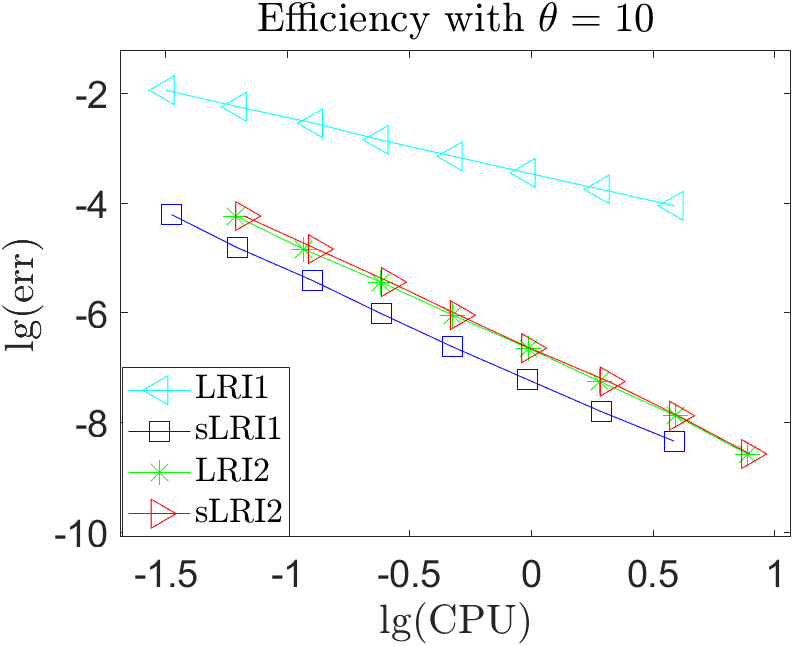}
         
    \end{minipage}
    
    \caption{Comparison of CPU time versus error under initial conditions in $\theta=1.8$  and $\theta=10$ .}
    \label{CPU 1.8 10}
\end{figure}

This observed computational advantage stems from the cancellation of higher-order terms during the symmetric construction process, clearly underscoring the efficiency merits of our designed symmetric two-step method.

\textbf{Long time conservations}.

Let 
\begin{equation}
    H^n(t):=\frac12\int_{\mathbb T}(|\partial_tu^n|^2+|\partial_xu^n|^2+2\cos u)\ \dd x
\end{equation}
be the global energy of (\ref{numerical}). We examine the long-time energy conservation properties of sLRI1 (\ref{sLRI1}) and sLRI2 (\ref{sLRI2}). In the numerical experiments, we set the time-step $h=0.1$ and the patial discrete mesh $M=2^7$. We use $\frac{H(T_{\mathrm{end}})-H(0)}{H(0)}$ as the relative errors of discrete energy over a long interval with initial data $(u^0(x)/10,v^0(x)/10)$ lying on different Sobolev spaces. 

We first assume $T_{\mathrm{end}} = 2000$. The comparison in figure \ref{1.5 0.1} and \ref{1.7 0.1} shows that, with the low regularity initial data $\theta=1.5$ and $\theta=1.7$, the symmetric scheme sLRI2 exhibits approximate energy conservation over long time scales, whereas the non-conservative scheme (\ref{LBY LR2})  loses long-term performance in the energy conservation after $t>300$. Similarly, sLRI1 also demonstrates long-time approximate energy conservation and its relative error is also smaller than that of LRI1.

Furthmore, according to the \cite[Theorem 3]{Cohen2008}, symmetric algorithms exhibit better long-time energy conservation behavior for smooth initial data. Figure \ref{energy smooth} displays the energy conservation properties of sLRI1 and sLRI2 under finial time $T_{\mathrm{end}} =5000$ and the smooth initial data
\begin{equation*}
    \begin{aligned}
        \tilde{u}_0(x) &= \frac{1}{10} \sqrt{\frac{2a}{b}} \,\mathrm{sech}(r x),\\
        \tilde{v}_0(x) &= \frac{c}{10}   \sqrt{\frac{2a}{b}} \,r\,\mathrm{sech}(r x) \tanh(r x),
    \end{aligned}
\end{equation*}
with \(r = \sqrt{\frac{a}{a^2 - c^2}}\). As can be seen, the symmetric methods maintain excellent energy conservation even over very long time scales.

\begin{figure}[t!]
    \centering
    \begin{minipage}{0.45\textwidth}
        \centering
        \includegraphics[width=\textwidth]{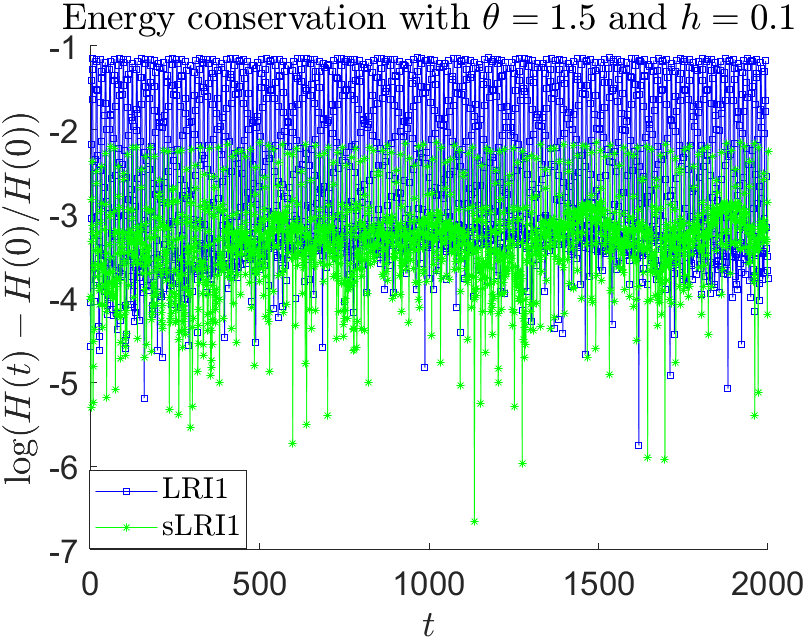}
         
    \end{minipage}
    \hfill
    \begin{minipage}{0.45\textwidth}
        \centering
        \includegraphics[width=\textwidth]{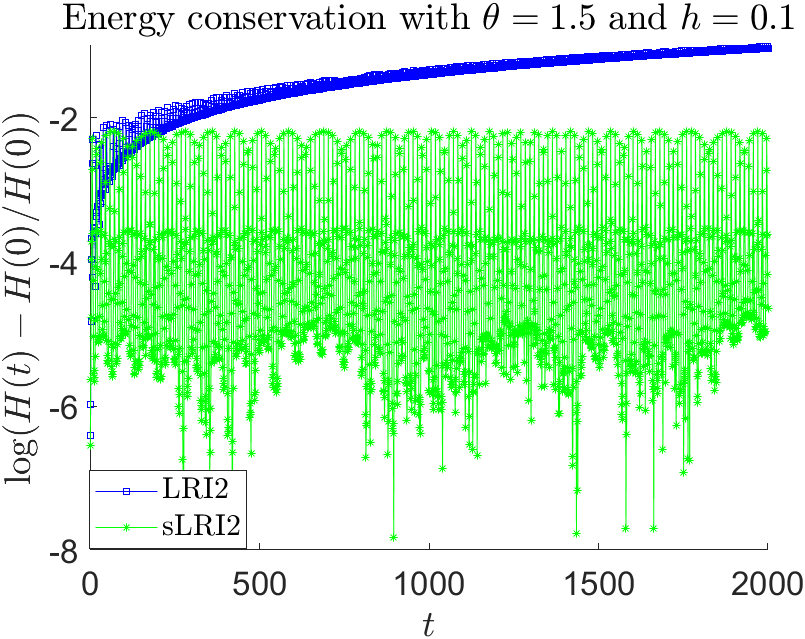}
          
    \end{minipage}
    
    \caption{Relative errors of sLRI1  and sLRI2  for (\ref{numerical}) with time step $h=0.1$, $T_{\mathrm{end}}=2000$ and $H^{1.5}(\Bbb T)\times H^{0.5}(\Bbb T)$ initial datum.}
    \label{1.5 0.1}
\end{figure}

\begin{figure}[t!]
    \centering
    \begin{minipage}{0.45\textwidth}
        \centering
        \includegraphics[width=\textwidth]{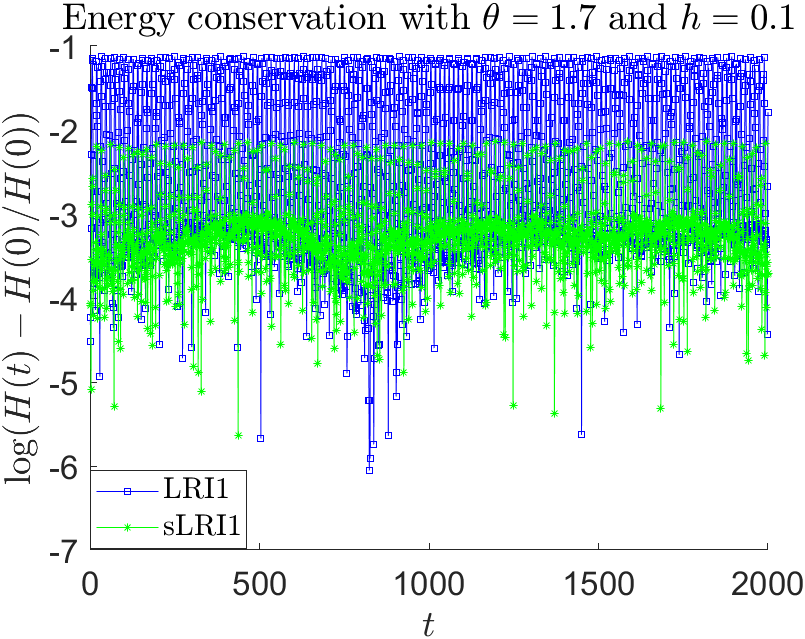}
         
    \end{minipage}
    \hfill
    \begin{minipage}{0.45\textwidth}
        \centering
        \includegraphics[width=\textwidth]{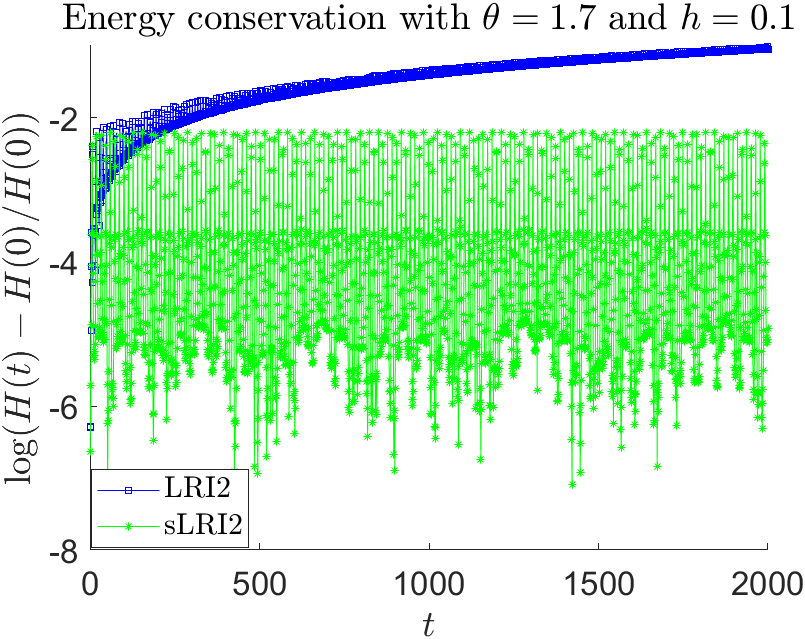}
          
    \end{minipage}
    
    \caption{Relative errors of sLRI1  and sLRI2  for (\ref{numerical}) with time step $h=0.1$, $T_{\mathrm{end}}=2000$ and $H^{1.7}(\Bbb T)\times H^{0.7}(\Bbb T)$ initial datum.}
    \label{1.7 0.1}
\end{figure}

\begin{figure}[t!]
    \centering
    \begin{minipage}{0.45\textwidth}
        \centering
        \includegraphics[width=\textwidth]{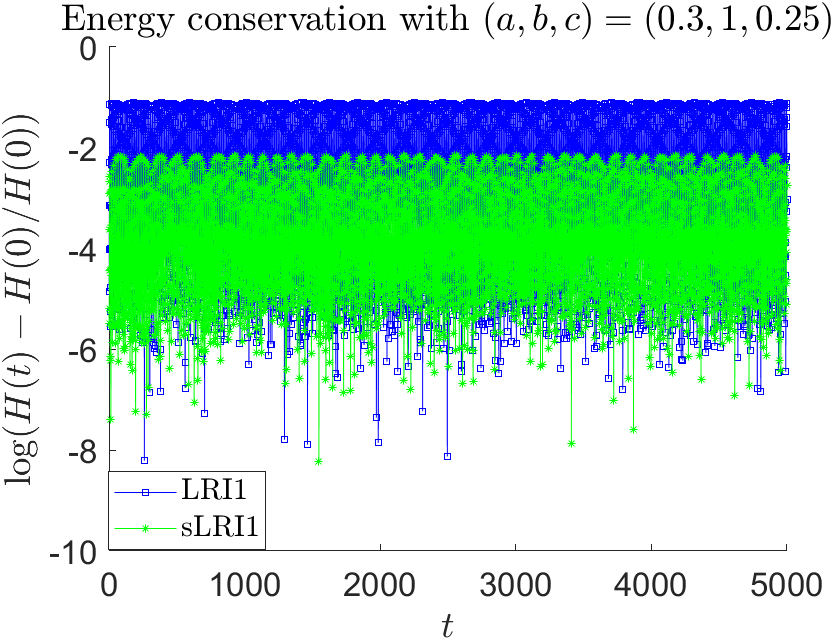}
          
    \end{minipage}
    \hfill
    % 右侧图片，手动添加 
    \begin{minipage}{0.45\textwidth}
        \centering
        \includegraphics[width=\textwidth]{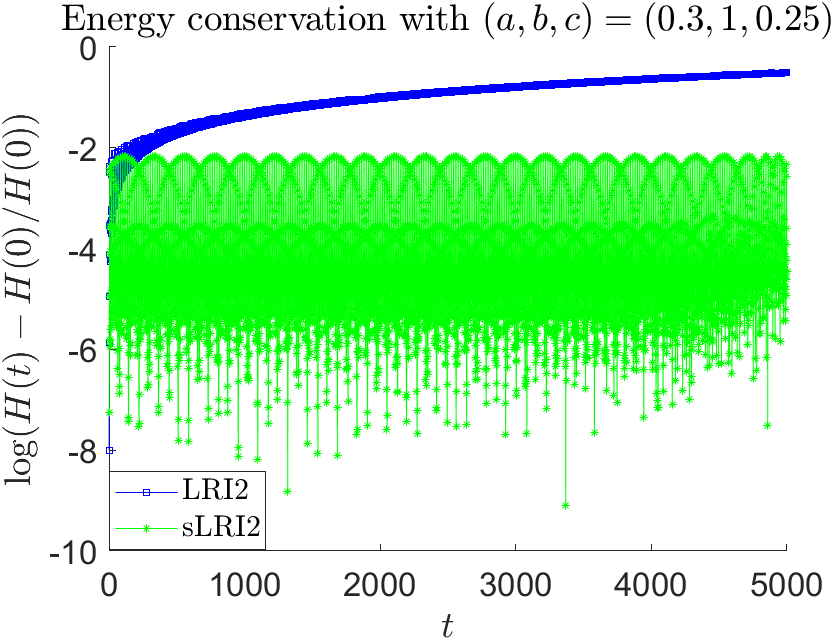}
          
    \end{minipage}
    
    \caption{Relative errors of sLRI1 and sLRI2  for (\ref{numerical}) with time step $h=0.1$, $T_{\mathrm{end}}=5000$, $(a,b,c)=(0.3,1,0.25)$ in $\tilde{u}_0(x)$ and $\tilde{v}_0(x)$.}
    \label{energy smooth}
\end{figure}

\section{Conclusion}
In this paper, we formulate and analyze a class of explicit symmetric low-regularity exponential integrator for the Klein-Gordon equation. This class of integrators maintains high accuracy and computational efficiency even when applied to initial data with low regularity. Furthermore, owing to its symmetry, the method exhibits excellent long-term energy conservation properties. We also demonstrate that the construction strategy for such symmetric exponential integrators is applicable to other nonlinear wave equations. Consequently, we believe that this approach paves the way for the development of a broader class of explicit symmetric low-regularity exponential integrators.

\section*{Competing interests}
We declare that we have no conflict of interest.
\section*{Funding}
This work was supported partially by the  National Natural Science Foundation of China (Grant No. 12371403).

%\clearpage
\bibliography{reference}% common bib file

\end{document}